\documentclass[11pt]{amsart}

\usepackage[margin=1.25in]{geometry}

\usepackage{graphicx}

\usepackage{amsmath, amssymb, amsthm, graphicx, enumerate, tikz, float, color}
\usepackage{mathtools, comment}
\usepackage{wasysym}
\usepackage[misc]{ifsym}
\usepackage[colorlinks]{hyperref}

\hypersetup{citecolor=blue}
\usetikzlibrary{matrix,arrows,decorations.pathmorphing}

\newtheorem{theorem}{Theorem}[section]
\newtheorem{lemma}[theorem]{Lemma}
\newtheorem{corollary}[theorem]{Corollary}
\newtheorem{proposition}[theorem]{Proposition}
\newtheorem{claim}[theorem]{Claim}
\newtheorem*{thmmain}{Theorem 1.1}
\newtheorem{hyp}[theorem]{Hypothesis}

\theoremstyle{definition}
\newtheorem{definition}[theorem]{Definition}

\theoremstyle{remark}
\newtheorem{remark}[theorem]{Remark}

\DeclareMathOperator{\cd}{cd}
\DeclareMathOperator{\Irr}{Irr}

\numberwithin{equation}{section}

\newcommand{\SL}[2]{\Sigma_{#1,#2}^L} 
\newcommand{\SR}[2]{\Sigma_{#1,#2}^R} 
\newcommand{\SLR}[2]{\Sigma_{#1,#2}^*} 
\newcommand{\SRi}[3]{\Sigma_{#1,#2}^{#3R}} 

\newcommand{\Mark}[1]{\textcolor{blue}{#1}}
\newcommand{\Jake}[1]{\textcolor{red}{#1}}
\newcommand{\Sara}[1]{\textcolor{purple}{#1}}

\makeatletter
\@namedef{subjclassname@2020}{%
  \textup{2020} Mathematics Subject Classification}
\makeatother

\begin{document}

\allowdisplaybreaks

\title[Prime character degree graphs within a family (ii)]{On prime character degree graphs occurring within a family
of graphs (ii)}

\author{Sara DeGroot}
\address{Department of Mathematics, St. Norbert College, De Pere, WI 54115}
\email{sara.degroot@snc.edu}

\author{Jacob Laubacher}
\address{Department of Mathematics, St. Norbert College, De Pere, WI 54115}
\email{jacob.laubacher@snc.edu}

\author{Mark Medwid}
\address{Department of Mathematical Sciences, Rhode Island College, Providence, RI 02908}
\email{mmedwid@ric.edu}


\date{\today}

\keywords{character degree graphs, solvable groups, normal nonabelian Sylow subgroups\\\indent\emph{Corresponding author.} Mark Medwid \Letter~\href{mailto:mmedwid@ric.edu}{mmedwid@ric.edu} \phone~401-456-9761.}

\begin{abstract}
In this paper, we continue the classification work done in the first paper of the same name. With careful modifications of our previous approach, we are able to deduce (with two notable exceptions) which members of the previously introduced graph family manifest as the prime character degree graph of some solvable group.
\end{abstract}


\maketitle

\section{Introduction}

Throughout this paper, we let $G$ be a finite solvable group. Following convention, we let $\Irr(G)$ denote the set of irreducible characters of $G$, and $\cd(G)=\{\chi(1)~:~\chi\in\Irr(G)\}$. The corresponding prime character degree graph for $G$ is a simple graph denoted $\Delta(G)$. Letting $\rho(\Gamma)$ represent the vertex set of a general graph $\Gamma$, we note that $\rho(\Delta(G))$ consists of all such prime divisors of $\cd(G)$. In this context, the idea of a vertex is synonymous with that of a prime, and as is common, we will abbreviate $\rho(\Delta(G))$ with $\rho(G)$. There is an edge between two distinct vertices $p,q\in\rho(G)$ if there exists some character $a\in\cd(G)$ such that $pq\mid a$. Prime character degree graphs have been studied for some time (for a broad overview, one can see \cite{L3}).

In this paper, we follow \cite{LM}. We recall the construction of the family of graphs from \cite{LM}, where the goal was to classify which graphs do or do not occur as the prime character degree graph of a solvable group. The family of graphs, denoted $\{\SLR{k}{n}\}$, has two variants: $\SL{k}{n}$ and $\SR{k}{n}$. We refer to both simultaneously by $\SLR{k}{n}$, and note that this graph is built by first arbitrarily choosing integers $k$ and $n$ such that $1 \leq n \leq k$, and then following several constructive rules. As such, the graph $\SLR{k}{n}$ consists of two distinct subgraphs $A$ and $B$, a fixed vertex $c$, and satisfies the following:
\begin{enumerate}[(i)]
    \item $A$ is a complete graph on $k$ vertices $a_1, a_2, \ldots, a_k$,
    \item $B$ is a complete graph on $k+n$ vertices $b_1,b_2, \ldots, b_k, \ldots, b_{k+n}$,
    \item $c \notin \rho(A)$ and $c \notin \rho(B)$,
    \item $\rho(A) \cap \rho(B) = \varnothing$,
    \item there is an edge between $a_i$ and $b_i$ for all $1 \leq i \leq k$,
    \item there is an edge between $a_i$ and $b_{k+i}$ for all $1 \leq i \leq n$,
    \item there is an edge between $c$ and $a_i$ for all $1 \leq i \leq k$ in $\SL{k}{n}$,
    \item there is an edge between $c$ and $b_i$ for all $1 \leq i \leq k+n$ in $\SR{k}{n}$,
    \item there are no edges in the graph $\SLR{k}{n}$ other than the edges described in (i)--(viii).
\end{enumerate}

Reading this graph $\SLR{k}{n}$ from left to right, the reason for the notation becomes apparent: $k$ counts the number of vertices on the left, $n$ represents how many distinct one-to-two edge mappings occur from $A$ to $B$, and the superscript $L$ or $R$ tells us where the fixed vertex $c$ resides. The main theorem of \cite{LM} completely classifies the graphs in the family $\{\SL{k}{n}\}$, whereas this paper (a direct sequel) has the goal to classify the family $\{\SR{k}{n}\}$. Our main result for this paper is as follows:

\begin{theorem}\label{Rfamilyresult}
The graph $\SR{k}{n}$ occurs as the prime character degree graph of a solvable group when $(k,n)=(1,1)$ (see Figure \ref{fig11R}), and possibly when $(k,n)\in\{(2,1),(2,2)\}$ (see Figure \ref{figRmaybe}). Otherwise $\SR{k}{n}$ does not occur as the prime character degree graph of any solvable group.
\end{theorem}

\begin{figure}[htb]
    \centering
$
\begin{tikzpicture}[scale=2]
\node (1a) at (0,.5) {$a_1$};
\node (11a) at (.75,1) {$b_1$};
\node (22a) at (.75,0) {$b_2$};
\node (33a) at (1.25,.5) {$c$};
\path[font=\small,>=angle 90]
(11a) edge node [right] {$ $} (22a)
(11a) edge node [right] {$ $} (33a)
(22a) edge node [right] {$ $} (33a)
(1a) edge node [right] {$ $} (11a)
(1a) edge node [right] {$ $} (22a);
\end{tikzpicture}
$
    \caption{The graph $\SR{1}{1}$}
    \label{fig11R}
\end{figure}

\begin{figure}[htb]
    \centering
$
\begin{tikzpicture}[scale=2]
\node (1b) at (0,1) {$a_1$};
\node (2b) at (0,0) {$a_2$};
\node (11b) at (1.25,1) {$b_1$};
\node (22b) at (1.25,0) {$b_2$};
\node (33b) at (.75,.5) {$b_3$};
\node (44b) at (1.75,.5) {$c$};
\path[font=\small,>=angle 90]
(1b) edge node [right] {$ $} (2b)
(11b) edge node [right] {$ $} (22b)
(11b) edge node [right] {$ $} (33b)
(11b) edge node [right] {$ $} (44b)
(22b) edge node [right] {$ $} (33b)
(22b) edge node [right] {$ $} (44b)
(33b) edge node [right] {$ $} (44b)
(1b) edge node [right] {$ $} (11b)
(2b) edge node [right] {$ $} (22b)
(1b) edge node [right] {$ $} (33b);
\node (1c) at (2.5,1) {$a_1$};
\node (2c) at (2.5,0) {$a_2$};
\node (11c) at (3.75,1) {$b_1$};
\node (22c) at (3.75,0) {$b_2$};
\node (33c) at (3.25,.75) {$b_3$};
\node (44c) at (3.25,.25) {$b_4$};
\node (55c) at (4.25,.5) {$c$};
\path[font=\small,>=angle 90]
(1c) edge node [right] {$ $} (2c)
(11c) edge node [right] {$ $} (22c)
(11c) edge node [right] {$ $} (33c)
(11c) edge node [right] {$ $} (44c)
(11c) edge node [right] {$ $} (55c)
(22c) edge node [right] {$ $} (33c)
(22c) edge node [right] {$ $} (44c)
(22c) edge node [right] {$ $} (55c)
(33c) edge node [right] {$ $} (44c)
(33c) edge node [right] {$ $} (55c)
(44c) edge node [right] {$ $} (55c)
(1c) edge node [right] {$ $} (11c)
(2c) edge node [right] {$ $} (22c)
(1c) edge node [right] {$ $} (33c)
(2c) edge node [right] {$ $} (44c);
\end{tikzpicture}
$
    \caption{The graphs $\SR{2}{1}$ and $\SR{2}{2}$}
    \label{figRmaybe}
\end{figure}

As we saw in \cite{LM}, moving the fixed vertex $c$ from the left to the right causes a noteworthy change and variation between the graphs, substantiating separate strategies and ultimately separate papers.

\section{Preliminaries}

All relevant preliminary information may be found with light detail in \cite{LM}. However, we shall recall some of the results and terms seeing frequent mention throughout this paper. In particular, we touch on the following main tools: P\'{a}lfy's condition, taming of graphs that arise as subgraphs of $\SR{k}{n}$, and the notion of (strongly) admissible vertices, which is tied to the issue of normal nonabelian Sylow $p$-subgroups.

\subsection{P\'{a}lfy's condition and other family classifications}


In order to successfully determine whether a graph in the family $\{\SR{k}{n}\}$ does or does not manifest as the prime character degree graph of some solvable group, we shall rely upon both classical and contemporary results. The following result by P\'{a}lfy is classical, and offers a simple way to ``rule out'' a particular graph from occurring as $\Delta(G)$ for some solvable $G$.

\begin{lemma}[P\'alfy's condition from \cite{P}]\label{PC}
Let $G$ be a solvable group and let $\pi$ be a set of primes contained in $\Delta(G)$. If $|\pi|=3$, then there exists an irreducible character of $G$ with degree divisible by at least two primes from $\pi$. (In other words, any three vertices of the prime character degree graph of a solvable group span at least one edge.)
\end{lemma}





We will also briefly recall the classification work done on a family with construction similar to the $\SLR{k}{n}$ graphs. The graph $\Gamma_{k,t}$ (with $k \geq t \geq 1$) consists of two complete graphs $A$ and $B$ on $k$ and $t$ vertices, respectively. Letting $\rho(A) = \{a_1, a_2, \ldots, a_k\}$ and $\rho(B) = \{b_1, b_2, \ldots, b_t\}$, we also have that $a_i$ is adjacent to $b_i$ for any $i$ satisfying $1 \leq i \leq t$, while the remaining $k-t$ vertices in $A$ are adjacent to no vertices in $B$. Some of these graph arise naturally as subgraphs of $\SLR{k}{n}$, so the following result from \cite{BissLaub} will be especially useful in our arguments.

\begin{theorem}\label{KT}\emph{(\cite{BissLaub})}
The graph $\Gamma_{k,t}$ occurs as the prime character degree graph of a solvable group precisely when $t=1$ or $k=t=2$. Otherwise $\Gamma_{k,t}$ does not occur as the prime character degree graph of any solvable group, nor does any connected proper subgraph of $\Gamma_{k,t}$ with the same vertex set whenever $k\geq t\geq2$.
\end{theorem}


Occasionally, we are faced with subgraphs of the $\SR{k}{n}$ graphs which have diameter three. These graphs were substantially tamed in \cite{Sass}. In particular, having a diameter of three in $\Delta(G)$ imposes limitations on how many vertices are distance three away from one another. Given a point $p$ that has one or more points of distance three away from $p$, there must be exponentially more points that are distance two or three away from $p$ when compared to points adjacent to $p$. These results offer a quick way to check that a subgraph of interest does not occur as $\Delta(G)$ for some solvable $G$, and so in these instances we turn to Theorems 2 or 4 from \cite{Sass}.




Finally, in some instances we may need to consult works done on disconnected graphs, in which case we turn to P\'{a}lfy's inequality from \cite{P2} or Theorem 5.5 from \cite{L2}, both of which are stated fully in \cite{LM}.

\subsection{Normal nonabelian Sylow subgroups}


As in \cite{LM}, a primary tool for arguing that a particular graph is not equal to $\Delta(G)$ for some solvable $G$ involves the absence of normal nonabelian Sylow subgroups. This builds off previous tools established by Bissler and Lewis.

In \cite{BL}, Bissler and Lewis construct a family of graphs that do not occur as the prime character degree graph of any solvable group. The authors make use of P\'{a}lfy's condition and also develop a new tool for determining whether a particular graph (that satisfies P\'{a}lfy's condition) manifests as $\Delta(G)$ for some solvable group $G$. To simplify the language somewhat, we refer to a graph $\Gamma$ as \emph{occurring} if $\Gamma = \Delta(G)$ for some solvable $G$ and \emph{non-occurring} otherwise.
\begin{definition}(\cite{BL})
Let $\Gamma$ be a graph and $p$ a vertex of $\Gamma$. Consider the following three conditions:
\begin{enumerate}[(i)]
  \item the subgraph of $\Gamma$ obtained by removing $p$ and all edges incident to $p$ is non-occurring, 
  \item all of the subgraphs of $\Gamma$ obtained by removing one or more of the edges incident to $p$ are non-occurring, 
  \item all of the subgraphs of $\Gamma$ obtained by removing $p$, the edges incident to $p$, and one or more of the edges between two adjacent vertices of $p$ are non-occurring. 
\end{enumerate}
If $p$ satisfies conditions (i) and (ii), then $p$ is said to be an \textbf{admissible} vertex. If $p$ satisfies all three conditions, then $p$ is said to be a \textbf{strongly admissible} vertex of $\Gamma$.
\end{definition}
The notion of (strong) admissibility offers one pathway to our desired conclusion through use of some technical lemmas establishing an absence of normal nonabelian Sylow subgroups, as follows:

\begin{lemma}\emph{(\cite{BL})}
Let $G$ be a solvable group, and suppose $p$ is an admissible vertex of $\Delta(G)$. For every proper normal subgroup $N$ of $G$, suppose that $\Delta(N)$ is a proper subgraph of $\Delta(G)$. Then $O^p(G)=G$.
\end{lemma}
\begin{lemma}\label{strong}\emph{(\cite{BL})}
Let $G$ be a solvable group and assume that $p$ is a prime whose vertex is a strongly admissible vertex of $\Delta(G)$. For every proper normal subgroup $N$ of $G$, suppose that $\Delta(G/N)$ is a proper subgraph of $\Delta(G)$. Then a Sylow $p$-subgroup of $G$ is not normal.
\end{lemma}
\begin{lemma}\label{pi}\emph{(\cite{BL})}
Let $\Gamma$ be a graph satisfying P\'alfy's condition. Let $q$ be a vertex of $\Gamma$, and denote $\pi$ to be the set of vertices of $\Gamma$ adjacent to $q$, and $\rho$ to be the set of vertices of $\Gamma$ not adjacent to $q$. Assume that $\pi$ is the disjoint union of nonempty sets $\pi_1$ and $\pi_2$, and assume that no vertex in $\pi_1$ is adjacent in $\Gamma$ to any vertex in $\pi_2$. Let $v$ be a vertex in $\pi_2$ adjacent to an admissible vertex $s$ in $\rho$. Furthermore, assume there exists another vertex $w$ in $\rho$ that is not adjacent to $v$.

Let $G$ be a solvable group such that $\Delta(G)=\Gamma$, and assume that for every proper normal subgroup $N$ of $G$, $\Delta(N)$ is a proper subgraph of $\Delta(G)$. Then a Sylow $q$-subgroup of $G$ for the prime associated to $q$ is not normal.
\end{lemma}
We shall employ the above lemmas to aid in the proof of our main result. Checking admissibility is done by direct computation and considering the necessary cases. This is often achieved through the use of several results, namely P\'{a}lfy's condition, Theorem \ref{KT}, and diameter three arguments from the results of \cite{Sass}. 

Once the issue of normal nonabelian Sylow subgroups is addressed, we can use the following to obtain a final contradiction:

\begin{lemma}\label{lemma3}\emph{(\cite{BL})}
Let $\Gamma$ be a graph satisfying P\'alfy's condition with $n\geq5$ vertices. Also, assume there exist distinct vertices $a$ and $b$ of $\Gamma$ such that $a$ is adjacent to an admissible vertex $c$, $b$ is not adjacent to $c$, and $a$ is not adjacent to an admissible vertex $d$.

Let $G$ be a solvable group and suppose for all proper normal subgroups $N$ of $G$ we have that $\Delta(N)$ and $\Delta(G/N)$ are proper subgraphs of $\Gamma$. Let $F$ be the Fitting subgroup of $G$ and suppose that $F$ is minimal normal in $G$. Then $\Gamma$ is not the prime character degree graph of any solvable group.
\end{lemma}

While this process is generally straightforward, the language of admissibility can sometimes become unwieldy in writing -- this is because we must consider a graph and a multitude of subgraphs with particular vertices or edges removed. So, we introduce the following notation to simplify the proof's writing and better identify subgraphs of interest. 

Let $\Gamma$ be a graph, and $p,q$ be vertices of $\Gamma$. We shall use the notation $\Gamma[p]$ to denote the subgraph of $\Gamma$ obtained by deleting $p$ and its incident edges from $\Gamma$. Since the graphs under consideration are devoid of multi-edges, we will use the expression $\epsilon(p,q)$ to refer to the edge joining $p$ and $q$ (and, naturally, $\epsilon(p,q) = \epsilon(q,p)$ since these graphs are not directed). Furthermore, we shall use $\Gamma[\epsilon(p,q)]$ to denote the subgraph of $\Gamma$ obtained by deleting the edge between $p$ and $q$, but not the vertices themselves. We can then denote multiple deletions by separating items inside the brackets with a comma: for instance, $\Gamma[p_1,\epsilon(p_2,p_3)]$ would be the subgraph of $\Gamma$ obtained by deleting $p_1$ and its incident edges, along with the edge connecting $p_2$ and $p_3$.

\subsection{Multiple special vertices}


As previously mentioned, a different approach is needed to classify the $\{\SR{k}{n}\}$ family with regards to which graphs manifest as $\Delta(G)$ for some solvable $G$. Specifically, we are required to take a necessary detour through a family with a construction near-identical to that of $\{\SR{k}{n}\}$. Recall that $\SR{k}{n}$ has a left-hand subgraph $A$, a right-hand subgraph $B$, and one ``special'' vertex $c$, in the sense that $c$ is adjacent only to vertices in $B$. In this paper, we will also consider a graph family with \emph{multiple} special vertices adjacent to the right-hand subgraph, which we will denote with $\{\SRi{k}{n}{m}\}$, where $m \geq 1$ and $1 \leq n \leq k$. The graph $\SRi{k}{n}{m}$ consists of three distinct subgraphs $A$, $B$, and $C$, and obeys the following:

\begin{enumerate}[(i)]
    \item the graphs $A$, $B$, and $C$ are complete graphs on $k$, $k+n$, and $m$ vertices, respectively,
    \item $\rho(A) = \{a_1, \ldots, a_k\}$, $\rho(B) = \{b_1, \ldots, b_k, \ldots, b_{k+n}\}$, and $\rho(C) = \{c_1, \ldots, c_m\}$ are all mutually  disjoint,
    \item the edge set of $\SRi{k}{n}{m}$ consists of all edges from $A$, $B$, and $C$, together with the following additional edges:
    \begin{itemize}
        \item $\epsilon(a_i,b_i)$ for all $1 \leq i \leq k$,
        \item $\epsilon(a_i,b_{k+i})$ for all $1 \leq i \leq n$,
        \item $\epsilon(b,c)$ for every $b \in \rho(B)$ and $c \in \rho(C)$.
    \end{itemize}
\end{enumerate}

\begin{remark}\label{mRrmk}
One could define a family of graphs $\{\Sigma_{k,n}^{mL}\}$ by replacing the third bullet point with ``$\epsilon(a,c)$ for every $a \in \rho(A)$ and $c \in \rho(C)$.'' However, consideration of this family was not necessary for classifying $\{\SL{k}{n}\}$, so we do not explore it any further in this paper.
\end{remark}

The above is a modification of the previous construction of the graph $\SR{k}{n}$ in that there are now precisely $m$ points incident to points in $B$ and not incident to points in $A$, as opposed to the single point added in $\SR{k}{n}$ (which is also $\SRi{k}{n}{1}$). These new graphs become important when we note that $\SRi{k}{i}{(n-i+1)}$ is a subgraph of $\SRi{k}{i+1}{(n-i)}$ with the same vertex set for $i$ satisfying $1 \leq i \leq n-2$.



\section{The family of graphs $\{\SR{k}{n}\}$}

The goal of this section is to prove our main theorem, stated again below:

\begin{thmmain}
The graph $\SR{k}{n}$ occurs as the prime character degree graph of a solvable group when $(k,n)=(1,1)$ (see Figure \ref{fig11R}), and possibly when $(k,n)\in\{(2,1),(2,2)\}$ (see Figure \ref{figRmaybe}). Otherwise $\SR{k}{n}$ does not occur as the prime character degree graph of any solvable group.
\end{thmmain}

It is actually quite easy to see that the graph $\SR{1}{1}$ occurs as the prime character degree graph of a solvable group; this result is well-established and can be found in \cite{H2} (or archived in \cite{L3}). The graphs $\SR{2}{1}$ and $\SR{2}{2}$ continue to elude us, however, 
because both have a subgraph with the same vertex set whose disconnected components do not violate P\'alfy's inequality from \cite{P2}. In fact, both of their disconnected subgraphs occur as $\Delta(G)$ for some solvable group $G$ where they have disconnected components of size $2$ and $4$, and of size $2$ and $5$, respectively. Finally, we note that $\SR{2}{1}$ is one of the nine unknown graphs that went unclassified in \cite{BL2}.


The proof technique of Theorem \ref{Rfamilyresult} closely follows that of the main theorem of \cite{LM}. While in broad strokes the process is the same, in the instance of the family $\{\SR{k}{n}\}$ the order in which claims and sub-claims are proved is more delicate. So, here we shall outline the overall proof's structure. 

At the top level, the proof of Theorem \ref{Rfamilyresult} is a proof by induction on $n$. This means that at every step $t$ in the inductive process, we are proving the non-occurrence of the family $\{\SR{k}{t}\}$. Note that for the graph $\SR{k}{t}$ to be defined, we need to have $k \geq t$; we shall henceforth operate under the assumption that, wherever denoted, $k$ satisfies the inequality $k \geq \max{\{t,3\}}$ (since the cases $k=1$ and $k=2$ were discussed above). 

Non-occurrence of the family $\{\SR{k}{1}\}$ then serves as the basis of induction. The proof of this step closely aligns with the basis step of the induction in \cite{LM}, so many details will be omitted. However, as $n$ increases, we see a significant departure from the argument of \cite{LM}. To prove the non-occurrence of $\SR{k}{2}$, for example, we must rely on the  non-occurrence of $\SR{k}{1}$ as well as the non-occurrence of $\SRi{k}{1}{2}$. This is due to our reliance on the tool of admissible vertices and, by extension, the non-occurrence of subgraphs.

Each individual stage of induction operates as a proof by contradiction -- that is, we assume that $\SR{k}{t} = \Delta(G)$ for some solvable $G$. We illustrate that many of the vertices are (strongly) admissible, and those that aren't satisfy Lemma \ref{pi} or the technical hypothesis found in \cite{BLL}. This leads to the conclusion that there is no normal nonabelian Sylow $q$-subgroup for any $q \in \rho(G)$. With some additional work, we are able to conclude that the Fitting subgroup of $G$ is minimal normal and then apply Lemma \ref{lemma3}.

Checking admissibility in the inductive step requires consideration of the subgraph $\SRi{k}{t}{2}$, and the non-occurrence of this graph is implied by the non-occurrence of $\SRi{k}{1}{(t+1)}$, the non-occurrence of which follows from the inductive hypothesis in a manner similar to that of the non-occurrence of $\SRi{k}{1}{2}$.

\subsection{The base case.}\label{BCsubsection}

As stated above, we omit many details from the proof of our base case, seen below. The arguments follow closely to what was done in \cite{LM}, and one can see Figure \ref{figk1R} for examples of graphs with $n=1$.

\begin{figure}[htb]
    \centering
$
\begin{tikzpicture}[scale=2]
\node (0a) at (1.75,.5) {$c$};
\node (1a) at (.5,.5) {$a_1$};
\node (11a) at (1.25,1) {$b_1$};
\node (22a) at (1.25,0) {$b_2$};
\path[font=\small,>=angle 90]
(0a) edge node [right] {$ $} (11a)
(0a) edge node [right] {$ $} (22a)
(11a) edge node [right] {$ $} (22a)
(1a) edge node [right] {$ $} (11a)
(1a) edge node [right] {$ $} (22a);
\node (0b) at (4.25,.5) {$c$};
\node (1b) at (2.5,1) {$a_1$};
\node (2b) at (2.5,0) {$a_2$};
\node (33b) at (3.25,.5) {$b_3$};
\node (11b) at (3.75,1) {$b_1$};
\node (22b) at (3.75,0) {$b_2$};
\path[font=\small,>=angle 90]
(0b) edge node [right] {$ $} (11b)
(0b) edge node [above] {$ $} (22b)
(0b) edge node [above] {$ $} (33b)
(1b) edge node [above] {$ $} (2b)
(11b) edge node [above] {$ $} (22b)
(11b) edge node [above] {$ $} (33b)
(22b) edge node [above] {$ $} (33b)
(1b) edge node [above] {$ $} (11b)
(2b) edge node [above] {$ $} (22b)
(1b) edge node [above] {$ $} (33b);
\node (0c) at (7.25,.5) {$c$};
\node (2c) at (5,1) {$a_2$};
\node (3c) at (5,0) {$a_3$};
\node (1c) at (5.5,.5) {$a_1$};
\node (11c) at (6.25,.75) {$b_1$};
\node (44c) at (6.25,.25) {$b_4$};
\node (22c) at (6.75,1) {$b_2$};
\node (33c) at (6.75,0) {$b_3$};
\path[font=\small,>=angle 90]
(0c) edge node [right] {$ $} (11c)
(0c) edge node [above] {$ $} (22c)
(0c) edge node [above] {$ $} (33c)
(0c) edge node [above] {$ $} (44c)
(1c) edge node [above] {$ $} (2c)
(1c) edge node [above] {$ $} (3c)
(2c) edge node [right] {$ $} (3c)
(11c) edge node [above] {$ $} (22c)
(11c) edge node [right] {$ $} (33c)
(11c) edge node [above] {$ $} (44c)
(22c) edge node [above] {$ $} (33c)
(22c) edge node [above] {$ $} (44c)
(33c) edge node [above] {$ $} (44c)
(1c) edge node [above] {$ $} (11c)
(2c) edge node [above] {$ $} (22c)
(3c) edge node [above] {$ $} (33c)
(1c) edge node [above] {$ $} (44c);
\node (1d) at (3.25,-1) {$c$};
\node (2d) at (.5,-.5) {$a_3$};
\node (3d) at (.5,-1.5) {$a_4$};
\node (4d) at (1,-.75) {$a_1$};
\node (5d) at (1,-1.25) {$a_2$};
\node (6d) at (2.25,-.5) {$b_3$};
\node (7d) at (2.25,-1.5) {$b_4$};
\node (8d) at (2.75,-.5) {$b_1$};
\node (9d) at (2.75,-1.5) {$b_2$};
\node (10d) at (1.75,-1) {$b_5$};
\path[font=\small,>=angle 90]
(1d) edge node [right] {$ $} (6d)
(1d) edge node [right] {$ $} (7d)
(1d) edge node [right] {$ $} (8d)
(1d) edge node [right] {$ $} (9d)
(1d) edge node [above] {$ $} (10d)
(2d) edge node [right] {$ $} (3d)
(2d) edge node [right] {$ $} (4d)
(2d) edge node [right] {$ $} (5d)
(3d) edge node [right] {$ $} (4d)
(3d) edge node [right] {$ $} (5d)
(4d) edge node [right] {$ $} (5d)
(6d) edge node [right] {$ $} (7d)
(6d) edge node [right] {$ $} (8d)
(6d) edge node [right] {$ $} (9d)
(6d) edge node [right] {$ $} (10d)
(7d) edge node [right] {$ $} (8d)
(7d) edge node [right] {$ $} (9d)
(7d) edge node [right] {$ $} (10d)
(8d) edge node [right] {$ $} (9d)
(8d) edge node [right] {$ $} (10d)
(9d) edge node [right] {$ $} (10d)
(2d) edge node [right] {$ $} (6d)
(3d) edge node [right] {$ $} (7d)
(4d) edge node [right] {$ $} (8d)
(5d) edge node [right] {$ $} (9d)
(4d) edge node [right] {$ $} (10d);
\node (1) at (7.25,-1) {$c$};
\node (2) at (4,-.75) {$a_4$};
\node (3) at (4,-1.25) {$a_5$};
\node (4) at (4.5,-.5) {$a_2$};
\node (5) at (4.5,-1.5) {$a_3$};
\node (6) at (5,-1) {$a_1$};
\node (7) at (5.75,-.75) {$b_1$};
\node (8) at (5.75,-1.25) {$b_6$};
\node (11) at (6.25,-.5) {$b_2$};
\node (12) at (6.25,-1.5) {$b_3$};
\node (9) at (6.75,-.575) {$b_4$};
\node (10) at (6.75,-1.425) {$b_5$};
\path[font=\small,>=angle 90]
(1) edge node [right] {$ $} (7)
(1) edge node [right] {$ $} (8)
(1) edge node [right] {$ $} (9)
(1) edge node [right] {$ $} (10)
(1) edge node [right] {$ $} (11)
(1) edge node [right] {$ $} (12)
(2) edge node [right] {$ $} (3)
(2) edge node [right] {$ $} (4)
(2) edge node [right] {$ $} (5)
(2) edge node [right] {$ $} (6)
(3) edge node [right] {$ $} (4)
(3) edge node [right] {$ $} (5)
(3) edge node [right] {$ $} (6)
(4) edge node [right] {$ $} (5)
(4) edge node [right] {$ $} (6)
(5) edge node [right] {$ $} (6)
(7) edge node [right] {$ $} (8)
(7) edge node [right] {$ $} (9)
(7) edge node [right] {$ $} (10)
(7) edge node [right] {$ $} (11)
(7) edge node [right] {$ $} (12)
(8) edge node [right] {$ $} (9)
(8) edge node [right] {$ $} (10)
(8) edge node [right] {$ $} (11)
(8) edge node [right] {$ $} (12)
(9) edge node [right] {$ $} (10)
(9) edge node [right] {$ $} (11)
(9) edge node [right] {$ $} (12)
(10) edge node [right] {$ $} (11)
(10) edge node [right] {$ $} (12)
(11) edge node [right] {$ $} (12)
(2) edge node [right] {$ $} (9)
(3) edge node [right] {$ $} (10)
(4) edge node [right] {$ $} (11)
(5) edge node [right] {$ $} (12)
(6) edge node [right] {$ $} (7)
(6) edge node [right] {$ $} (8);
\end{tikzpicture}
$
    \caption{Examples of graphs in the family $\{\SR{k}{1}\}$: $1\leq k\leq5$}
    \label{figk1R}
\end{figure}

\begin{proposition}\label{n1case}
Let $k\geq3$. The graph $\SR{k}{1}$ is not the prime character degree graph of any solvable group.
\end{proposition}
\begin{proof}
For the sake of contradiction, suppose that there exists a solvable group $G$ with $|G|$ minimal such that $\Delta(G)=\SR{k}{1}$. Following the proof of the base case in \cite{LM}, one can show that $b_i$ is strongly admissible for all $1\leq i\leq k+1$ (using the subgraphs $\Gamma_{k+1,k}$ and $\Gamma_{k+2,k}$, or diameter three arguments from \cite{Sass}). One can further show that $a_j$ satisfies Lemma \ref{pi} for all $1\leq j\leq k$, and that the graph $\SR{k}{1}$ satisfies the technical hypothesis from \cite{BLL} under the notation for $p=c$. In total, we now have that there is no normal nonabelian Sylow $q$-subgroup for any vertex $q\in\rho(G)$.

Next, one can go through steps to conclude that the Frattini subgroup $\Phi(G)=1$. Then one gets the existence of a subgroup $H$ of $G$ such that $G=HF$ and $H\cap F=1$, where $F$ is the Fitting subgroup of $G$. One can then go through the exhaustive steps to check that $F$ is minimal normal in $G$, and that the conditions of Lemma \ref{lemma3} are satisfied. This yields the desired result.
\end{proof}

As consequence of Proposition \ref{n1case}, and needed for the family $\{\SR{k}{2}\}$, we now turn our attention to the family $\{\SRi{k}{1}{2}\}$. As always, we take $k\geq3$, but for examples of these graphs in general, one can see Figure \ref{figk1RR}.

\begin{figure}[htb]
    \centering
$
\begin{tikzpicture}[scale=2]
\node (0a) at (1.75,.75) {$c_1$};
\node (00a) at (1.75,.25) {$c_2$};
\node (1a) at (.5,.5) {$a_1$};
\node (11a) at (1.25,1) {$b_1$};
\node (22a) at (1.25,0) {$b_2$};
\path[font=\small,>=angle 90]
(0a) edge node [right] {$ $} (00a)
(0a) edge node [right] {$ $} (11a)
(0a) edge node [right] {$ $} (22a)
(00a) edge node [right] {$ $} (11a)
(00a) edge node [right] {$ $} (22a)
(11a) edge node [right] {$ $} (22a)
(11a) edge node [right] {$ $} (22a)
(1a) edge node [right] {$ $} (11a)
(1a) edge node [right] {$ $} (22a);
\node (0b) at (4.25,.75) {$c_1$};
\node (00b) at (4.25,.25) {$c_2$};
\node (1b) at (2.5,1) {$a_1$};
\node (2b) at (2.5,0) {$a_2$};
\node (33b) at (3.25,.5) {$b_3$};
\node (11b) at (3.75,1) {$b_1$};
\node (22b) at (3.75,0) {$b_2$};
\path[font=\small,>=angle 90]
(00b) edge node [right] {$ $} (0b)
(00b) edge node [right] {$ $} (11b)
(00b) edge node [right] {$ $} (22b)
(00b) edge node [right] {$ $} (33b)
(0b) edge node [right] {$ $} (11b)
(0b) edge node [above] {$ $} (22b)
(0b) edge node [above] {$ $} (33b)
(1b) edge node [above] {$ $} (2b)
(11b) edge node [above] {$ $} (22b)
(11b) edge node [above] {$ $} (33b)
(22b) edge node [above] {$ $} (33b)
(1b) edge node [above] {$ $} (11b)
(2b) edge node [above] {$ $} (22b)
(1b) edge node [above] {$ $} (33b);
\node (0c) at (7.25,.75) {$c_1$};
\node (00c) at (7.25,.25) {$c_2$};
\node (2c) at (5,1) {$a_2$};
\node (3c) at (5,0) {$a_3$};
\node (1c) at (5.5,.5) {$a_1$};
\node (11c) at (6.25,.75) {$b_1$};
\node (44c) at (6.25,.25) {$b_4$};
\node (22c) at (6.75,1) {$b_2$};
\node (33c) at (6.75,0) {$b_3$};
\path[font=\small,>=angle 90]
(00c) edge node [right] {$ $} (0c)
(00c) edge node [right] {$ $} (11c)
(00c) edge node [right] {$ $} (22c)
(00c) edge node [right] {$ $} (33c)
(00c) edge node [right] {$ $} (44c)
(0c) edge node [right] {$ $} (11c)
(0c) edge node [above] {$ $} (22c)
(0c) edge node [above] {$ $} (33c)
(0c) edge node [above] {$ $} (44c)
(1c) edge node [above] {$ $} (2c)
(1c) edge node [above] {$ $} (3c)
(2c) edge node [right] {$ $} (3c)
(11c) edge node [above] {$ $} (22c)
(11c) edge node [right] {$ $} (33c)
(11c) edge node [above] {$ $} (44c)
(22c) edge node [above] {$ $} (33c)
(22c) edge node [above] {$ $} (44c)
(33c) edge node [above] {$ $} (44c)
(1c) edge node [above] {$ $} (11c)
(2c) edge node [above] {$ $} (22c)
(3c) edge node [above] {$ $} (33c)
(1c) edge node [above] {$ $} (44c);
\node (0d) at (3.25,-.75) {$c_1$};
\node (1d) at (3.25,-1.25) {$c_2$};
\node (2d) at (.5,-.5) {$a_3$};
\node (3d) at (.5,-1.5) {$a_4$};
\node (4d) at (1,-.75) {$a_1$};
\node (5d) at (1,-1.25) {$a_2$};
\node (6d) at (2.25,-.5) {$b_3$};
\node (7d) at (2.25,-1.5) {$b_4$};
\node (8d) at (2.75,-.5) {$b_1$};
\node (9d) at (2.75,-1.5) {$b_2$};
\node (10d) at (1.75,-1) {$b_5$};
\path[font=\small,>=angle 90]
(0d) edge node [right] {$ $} (1d)
(0d) edge node [right] {$ $} (6d)
(0d) edge node [right] {$ $} (7d)
(0d) edge node [right] {$ $} (8d)
(0d) edge node [right] {$ $} (9d)
(0d) edge node [right] {$ $} (10d)
(1d) edge node [right] {$ $} (6d)
(1d) edge node [right] {$ $} (7d)
(1d) edge node [right] {$ $} (8d)
(1d) edge node [right] {$ $} (9d)
(1d) edge node [above] {$ $} (10d)
(2d) edge node [right] {$ $} (3d)
(2d) edge node [right] {$ $} (4d)
(2d) edge node [right] {$ $} (5d)
(3d) edge node [right] {$ $} (4d)
(3d) edge node [right] {$ $} (5d)
(4d) edge node [right] {$ $} (5d)
(6d) edge node [right] {$ $} (7d)
(6d) edge node [right] {$ $} (8d)
(6d) edge node [right] {$ $} (9d)
(6d) edge node [right] {$ $} (10d)
(7d) edge node [right] {$ $} (8d)
(7d) edge node [right] {$ $} (9d)
(7d) edge node [right] {$ $} (10d)
(8d) edge node [right] {$ $} (9d)
(8d) edge node [right] {$ $} (10d)
(9d) edge node [right] {$ $} (10d)
(2d) edge node [right] {$ $} (6d)
(3d) edge node [right] {$ $} (7d)
(4d) edge node [right] {$ $} (8d)
(5d) edge node [right] {$ $} (9d)
(4d) edge node [right] {$ $} (10d);
\node (0) at (7.25,-.75) {$c_1$};
\node (1) at (7.25,-1.25) {$c_2$};
\node (2) at (4,-.75) {$a_4$};
\node (3) at (4,-1.25) {$a_5$};
\node (4) at (4.5,-.5) {$a_2$};
\node (5) at (4.5,-1.5) {$a_3$};
\node (6) at (5,-1) {$a_1$};
\node (7) at (5.75,-.75) {$b_1$};
\node (8) at (5.75,-1.25) {$b_6$};
\node (11) at (6.25,-.5) {$b_2$};
\node (12) at (6.25,-1.5) {$b_3$};
\node (9) at (6.75,-.5) {$b_4$};
\node (10) at (6.75,-1.5) {$b_5$};
\path[font=\small,>=angle 90]
(0) edge node [right] {$ $} (1)
(0) edge node [right] {$ $} (7)
(0) edge node [right] {$ $} (8)
(0) edge node [right] {$ $} (9)
(0) edge node [right] {$ $} (10)
(0) edge node [right] {$ $} (11)
(0) edge node [right] {$ $} (12)
(1) edge node [right] {$ $} (7)
(1) edge node [right] {$ $} (8)
(1) edge node [right] {$ $} (9)
(1) edge node [right] {$ $} (10)
(1) edge node [right] {$ $} (11)
(1) edge node [right] {$ $} (12)
(2) edge node [right] {$ $} (3)
(2) edge node [right] {$ $} (4)
(2) edge node [right] {$ $} (5)
(2) edge node [right] {$ $} (6)
(3) edge node [right] {$ $} (4)
(3) edge node [right] {$ $} (5)
(3) edge node [right] {$ $} (6)
(4) edge node [right] {$ $} (5)
(4) edge node [right] {$ $} (6)
(5) edge node [right] {$ $} (6)
(7) edge node [right] {$ $} (8)
(7) edge node [right] {$ $} (9)
(7) edge node [right] {$ $} (10)
(7) edge node [right] {$ $} (11)
(7) edge node [right] {$ $} (12)
(8) edge node [right] {$ $} (9)
(8) edge node [right] {$ $} (10)
(8) edge node [right] {$ $} (11)
(8) edge node [right] {$ $} (12)
(9) edge node [right] {$ $} (10)
(9) edge node [right] {$ $} (11)
(9) edge node [right] {$ $} (12)
(10) edge node [right] {$ $} (11)
(10) edge node [right] {$ $} (12)
(11) edge node [right] {$ $} (12)
(2) edge node [right] {$ $} (9)
(3) edge node [right] {$ $} (10)
(4) edge node [right] {$ $} (11)
(5) edge node [right] {$ $} (12)
(6) edge node [right] {$ $} (7)
(6) edge node [right] {$ $} (8);
\end{tikzpicture}
$
    \caption{Examples of graphs in the family $\{\SRi{k}{1}{2}\}$: $1\leq k\leq5$}
    \label{figk1RR}
\end{figure}

Since this is where the argument becomes more delicate and deviates more significantly from that in \cite{LM}, here we will provide sufficient detail. The bones of the argument stay the same, however, so the process should be familiar. To wit, the following:

\begin{lemma}\label{admissiblevertices}
Let $k \geq 3$ and assume $\SRi{k}{1}{2} = \Delta(G)$ for some finite solvable $G$, where $|G|$ is minimal. Then $G$ does not have a normal nonabelian Sylow $p$-subgroup for any $p \in \{b_1, b_{k+1} \} \cup \{b_2, \ldots, b_k\} \cup \{c_1, c_2\}$. In particular, $p$ is a strongly admissible vertex.
\end{lemma}
\begin{proof}
The set of vertices $\{b_1, b_2, \ldots, b_{k+1}, c_1, c_2\}$ has been partitioned as above to indicate which vertices share (with appropriate relabeling) identical proofs of admissibility. Therefore, it suffices to explicitly illustrate strong admissibility for the representatives $p = b_1$, $p = b_2$, and $p = c_1$. 


We first consider admissibility of $b_1$. Now, $\SRi{k}{1}{2}[b_1]$ is isomorphic to $\Gamma_{k+2,k}$ from \cite{BissLaub}, and hence is non-occurring (see Theorem \ref{KT}). Next, $\SRi{k}{1}{2}[\epsilon(a_1,b_1)]$ is $\Gamma_{k+3,k}$, so again by Theorem \ref{KT} is non-occurring. It is easy to see that $\SRi{k}{1}{2}[\epsilon(b_1,b_i)]$ (where $2 \leq i \leq k+1$), $\SRi{k}{1}{2}[\epsilon(b_1,c_1)]$, and $\SRi{k}{1}{2}[\epsilon(b_1,c_2)]$ all violate P\'{a}lfy's condition. Therefore, $b_1$ is admissible. For strong admissibility, we consider $\SRi{k}{1}{2}[b_1]$ and removal of one or more edges between two vertices adjacent to $b_1$. 
Note that $\SRi{k}{1}{2}[b_1,\epsilon(a_1,b_{k+1})]$ has diameter three, and violates the main result from \cite{Sass}. Furthermore, we obtain an odd-length cycle in the complement graph if we delete an edge with both vertices in $\{b_2, \ldots, b_k, b_{k+1}, c_1, c_2\}$. Setting $c_1 = b_{k+2}$ and $c_2 = b_{k+3}$ for convenience of notation, removing an edge joining $b_i$ and $b_j$ (with $2\leq i<j\leq k+3$) gives $b_i, b_j,$ and $a_l$ having no common incident edges, where $l$ is chosen to satisfy $1 \leq l \leq k$, $l \neq i$, and $l \neq j$. Hence, $b_1$ is strongly admissible.

For admissibility of $b_2$, we now consider the $\SRi{k}{1}{2}[b_2]$, which has diameter three. Then, in the parlance of the main result from \cite{Sass}, we conclude the graph is non-occurring. Next, we consider retaining $b_2$ but deleting one or more of its incident edges. The graph $\SRi{k}{1}{2}[\epsilon(b_2,a_2)]$ has diameter three and is again non-occurring. Observe that $\SRi{k}{1}{2}[\epsilon(b_2,b_i)]$ (where $1 \leq i \leq k+1$ and $i \neq 2$) violates P\'{a}lfy's condition, as do $\SRi{k}{1}{2}[\epsilon(b_2,c_1)]$ and $\SRi{k}{1}{2}[\epsilon(b_2,c_2)]$. Thus, $b_2$ is admissible. To see that $b_2$ is strongly admissible, observe that the only adjacent vertices that are both mutually adjacent to $b_2$ are from the set $\{b_1, b_3, \ldots, b_{k+1}, c_1, c_2\}$. As such, one may follow the argument above for the similar case showing $b_1$ is strongly admissible.

Finally, we will illustrate that $c_1$ is strongly admissible. To that end, we first observe that $\SRi{k}{1}{2}[c_1]$ is $\SR{k}{1}$, and hence is non-occurring by Proposition \ref{n1case}. Furthermore, $\SRi{k}{1}{2}[\epsilon(c_1,b_i)]$ (for $1 \leq i \leq k+1$) violates P\'{a}lfy's condition, as does $\SRi{k}{1}{2}[\epsilon(c_1,c_2)]$, so $c_1$ is admissible. As was the case with $b_2$, the only adjacent vertices both mutually adjacent to $c_1$ lie in the right subgraph; hence, one may follow the same argument as was used with the strong admissibility of $b_1$ and $b_2$.
\end{proof}

As consequence of Lemma \ref{admissiblevertices}, the following lemma specifically investigates proper connected subgraphs of $\SRi{k}{1}{2}$.

\begin{lemma}\label{subgraphpart}
Let $k\geq3$ and consider the graph $\SRi{k}{1}{2}$. Letting $p=c_1$, we have the corresponding set of adjacent vertices $\pi=\{b_1,\ldots,b_{k+1},c_2\}$ and the set of nonadjacent vertices $\rho=\{a_1,\ldots,a_k\}$. Then letting $\pi^*$ be any nonempty subset of $\pi$, we have that no proper connected subgraph with vertex set $\{p\}\cup\pi^*\cup\rho$ occurs as the prime character degree graph of any solvable group.
\end{lemma}
\begin{proof}
First we notice that if $c_2\notin\pi^*$, then the resulting vertex set of $\{p\}\cup\pi^*\cup\rho$ will generate a subgraph of $\SR{k}{1}$, and it has already been established that no proper connected subgraph of $\SR{k}{1}$ with its corresponding vertex set $\{p\}\cup\pi^*\cup\rho$ occurs as the prime character degree graph of any solvable group. Hence, when considering subsets of $\pi$, we can now assume that $c_2\in\pi^*$.

Moreover, if $b_1\notin\pi^*$ or $b_{k+1}\notin\pi^*$, then the resulting graph generated by the vertex set $\{p\}\cup\pi^*\cup\rho$ is a subgraph of $\Gamma_{k,|\{p\}\cup\pi^*|}$, which we know does not occur as the prime character degree graph of any solvable group due to Theorem \ref{KT} (since $k\geq3$ and $|\{p\}\cup\{\pi^*\}|\geq2$, having already deduced that $c_2\in\pi^*$). Hence, we may assume that $b_1,b_{k+1}\in\pi^*$.

Finally, now knowing that $b_1,b_{k+1},c_2\in\pi^*$, we can then investigate the following cases, noting that the remaining vertices of $b_2,\ldots,b_k$ are all symmetric. Therefore, without loss of generality, we set the notation $\pi_i^*=\{b_2,\ldots,b_i\}$ for $2\leq i\leq k$ and consider the following possibilities for subsets of $\pi$: (a) $\{b_1,b_{k+1},c_2\}$, (b) $\{b_1,b_{k+1},c_2\}\cup\pi_i^*$, and (c) $\pi$.

For (a), we need only consider the case where we have both the edge from $a_1$ to $b_1$ and the edge from $a_1$ to $b_{k+1}$, for otherwise, if we did not have one of those edges, we can once again invoke Theorem \ref{KT}. Observe that the resulting graph has diameter three and violates the main result from \cite{Sass}.

For (b), notice that we need only consider the case where $2\leq i\leq k-1$ since the case $i=k$ is exactly (c) below. Note that this is once again a subgraph of $\SR{k}{1}$ and has already been shown not to occur.

For (c), it once again suffices to only consider losing the edge between $a_1$ and $b_1$ or the edge between $a_1$ and $b_{k+1}$. The resulting graph is a subgraph of $\Gamma_{k+3,k}$, which is handled by Theorem \ref{KT}.

Hence, no proper connected subgraph of $\SRi{k}{1}{2}$ with vertex set $\{p\}\cup\pi^*\cup\rho$ occurs as the prime character degree graph of any solvable group.
\end{proof}

Turning our attention to the remaining vertices, we have the following:

\begin{lemma}\label{pivertices}
Let $k \geq 3$ and assume $\SRi{k}{1}{2} = \Delta(G)$ for some finite solvable group $G$ with $|G|$ minimal. Then $G$ does not have a normal nonabelian Sylow $q$-subgroup for any $q \in \{a_1\} \cup \{a_2, \ldots, a_k\}$.
\end{lemma}
\begin{proof}
By symmetry, it suffices to illustrate the result for only $q=a_1$ and $q=a_2$. For $q=a_1$ and following the notation laid out in Lemma \ref{pi}, we have $\pi = \{b_1, b_{k+1}, a_2, \ldots, a_k\}$ and $\rho = \{b_2, \ldots, b_k,c_1,c_2\}$. Setting $\pi_1 = \{b_1,b_{k+1}\}$ and $\pi_2 = \{a_2, \ldots, a_k\}$, we may choose $v = a_2$, which is adjacent to $s = b_2$ in $\rho$, shown to be admissible in Lemma \ref{admissiblevertices}. We may also note that $w = c_1 \in \rho$ is not adjacent to $v$, hence the result follows by Lemma \ref{pi}.

Similarly, for $q=a_2$ we have $\pi = \{b_2, a_1, a_3, \ldots, a_k\}$ and $\rho = \{b_1, b_3, \ldots, b_{k+1},c_1,c_2\}$. Setting $\pi_1 = \{b_2\}$ and $\pi_2 = \{a_1, a_3, \ldots, a_k\}$, we take $v = a_1$, $s = b_1$, and $w = c_1$, whence the result via Lemma \ref{pi}.
\end{proof}

Under the assumption that there exists some solvable group $G$ with $|G|$ minimal such that $\Delta(G)=\SRi{k}{1}{2}$, we now have that there is no normal nonabelian Sylow $q$-subgroup for every $q\in\rho(G)$ and for all $k\geq3$. Letting $F$ be the Fitting subgroup of $G$, we note that $\rho(G)=\pi(|G:F|)$. Thus, $\rho(G)=\rho(G/\Phi(G))$, where $\Phi(G)$ is the Frattini subgroup of $G$. But we note that Lemma \ref{subgraphpart} verifies there is no proper connected subgraph of $\Delta(G)$ that occurs as the prime character degree graph of a solvable group. Furthermore, we notice that the disconnected subgraph with the same vertex set does not occur, since it violates P\'alfy's inequality from \cite{P2}, where the two components are of size $a=k$ and $b=k+3$ where $a\leq b$, yet $b=k+3<2^k-1=2^a-1$. Finally, since $|G|$ is minimal, we are forced to conclude that $\Phi(G)=1$. We can then apply Lemma III 4.4 from \cite{H} to get the existence of a subgroup $H$ of $G$ such that $G=HF$ and $H\cap F=1$.

\begin{lemma}\label{minimalnormal}
Let $k\geq3$ and assume $\SRi{k}{1}{2}=\Delta(G)$ for some finite solvable group $G$, where $|G|$ is minimal. Then the Fitting subgroup of $G$, denoted $F$, is minimal normal in $G$.
\end{lemma}
\begin{proof}
We prescribe to the same notation as in Lemma \ref{subgraphpart}. That is, set $p=c_1$, and then we have that $\pi=\{b_1,\ldots,b_{k+1},c_2\}$ and $\rho=\{a_1,\ldots,a_k\}$. We let $F$ be the Fitting subgroup of $G$, and defining $H$ as above, we set $E$ to be the Fitting subgroup of $H$.

Next, we proceed by contradiction. Therefore, we suppose there exists a normal subgroup $N$ of $G$ so that $1<N<F$. Notice that by Lemma III 4.5 of \cite{H} we know there exists some normal subgroup $M$ of $G$ such that $F=N\times M$. Furthermore, observe that $\rho(G/N)\subset\rho(G)$ and $\rho(G/M)\subset\rho(G)$ since $N$ and $M$ are both nontrivial. For any $q\in\rho(G)\setminus\rho(G/N)$, it is known that $G/N$ must have a normal nonabelian Sylow $q$-subgroup, the class of which is a formation, and so it is forced that $q\in\rho(G/M)$. Hence, we get that $\rho(G)=\rho(G/N)\cup\rho(G/M)$.

One can follow the arguments identical to that in \cite{BL}, \cite{LM}, or \cite{LD3} to conclude that $\rho(G)\setminus(\rho(G/N)\cap\rho(G/M))$ must lie in a complete subgraph of $\Delta(G)$. Therefore, $\rho(G)\setminus(\rho(G/N)\cap\rho(G/M))$ must lie in one of the following sets: (a) $\{p\}\cup\pi$, (b) $\{a_i,b_i\}$ for $2\leq i\leq k$, (c) $\{a_1,b_1,b_{k+1}\}$, or (d) $\rho$.

Suppose (a) occurs; that is, suppose $\rho(G)\setminus(\rho(G/N)\cap\rho(G/M))\subseteq\{p\}\cup\pi$. This implies that $\rho\subseteq\rho(G/N)\cap\rho(G/M)$. Since $\rho(G)=\rho(G/N)\cup\rho(G/M)$, we get that $p\in\rho(G/N)$ or $p\in\rho(G/M)$. Without loss of generality, suppose $p\in\rho(G/N)$. Therefore, we now have that $\{p\}\cup\rho\subseteq\rho(G/N)$. There are then two cases to consider for the set $\rho(G/N)$: either (i) $\rho(G/N)=\{p\}\cup\rho$, or (ii) $\rho(G/N)=\{p\}\cup\pi^*\cup\rho$ where $\pi^*$ is a nonempty proper subset of $\pi$. In case (i), the only possible graph that can arise is the disconnected graph with components $\{p\}$ and $\rho$ as the vertex sets. Using Theorem 5.5 from \cite{L2}, we get that $G/N$ has a central Sylow $b$-subgroup for some $b\in\pi$. As a consequence, we then get that $O^b(G)<G$, but this is a contradiction since all the vertices in $\pi$ are admissible by Lemma \ref{admissiblevertices}. In case (ii), we know that no connected graph with vertex set $\{p\}\cup\pi^*\cup\rho$ occurs as the prime character degree graph of any solvable group by Lemma \ref{subgraphpart}. Once again, the only option is for the disconnected subgraph to occur, leaving $G/N$ with a central Sylow $b$-subgroup with $b\in\pi\setminus\pi^*$. As above, this implies $O^b(G)<G$, a contradiction since $b$ is admissible.

Suppose (b) occurs; that is, suppose $\rho(G)\setminus(\rho(G/N)\cap\rho(G/M))\subseteq\{a_i,b_i\}$. This implies that $\rho(G)\setminus\{a_i,b_i\}\subseteq\rho(G/N)\cap\rho(G/M)$. Without loss of generality, let $a_i\in\rho(G/N)$. Since $\rho(G/N)$ and $\rho(G/M)$ are both proper in $\rho(G)$, we know that $b_i\notin\rho(G/N)$. Specifically, we get that $\rho(G/N)=\rho(G)\setminus\{b_i\}=\{p\}\cup\pi^*\cup\rho$ for the particular $\pi^*=\pi\setminus\{b_i\}$. Once again, no proper connected subgraph with the vertex set $\{p\}\cup\pi^*\cup\rho$ occurs, and therefore the only option is for the disconnected subgraph to occur. Again by Theorem 5.5 in \cite{L2}, we have that $G/N$ must have a central Sylow $b_i$-subgroup, giving us that $O^{b_i}(G)<G$. This is a contradiction since $O^{b_i}(G)=G$ because $b_i$ is admissible by Lemma \ref{admissiblevertices}.

Suppose (c) occurs; that is, suppose $\rho(G)\setminus(\rho(G/N)\cap\rho(G/M))\subseteq\{a_1,b_1,b_{k+1}\}$. This implies that $\rho(G)\setminus\{a_1,b_1,b_{k+1}\}\subseteq\rho(G/N)\cap\rho(G/M)$. Without loss of generality, we can suppose that $a_1\in\rho(G/N)$, which consequently gives three cases for $\rho(G/N)$: either (i) $\rho(G/N)=\rho(G)\setminus\{b_1,b_{k+1}\}$, or (ii) $\rho(G/N)=\rho(G)\setminus\{b_1\}$, or (iii) $\rho(G/N)=\rho(G)\setminus\{b_{k+1}\}$. In case (i), we know that no proper connected subgraph with this vertex set occurs as the prime character degree graph of any solvable group since $\rho(G)\setminus\{b_1,b_{k+1}\}=\{p\}\cup\pi^*\cup\rho$ for the particular $\pi^*=\pi\setminus\{b_1,b_{k+1}\}$. Therefore, following a similar argument as to the above, the disconnected subgraph must occur, which yields a central Sylow $b_i$-subgroup in $G/N$ and thus $O^{b_i}(G)<G$ for some $i=1$ or $i=k+1$, a contradiction since $b_i$ has been shown to be admissible. In case (ii), we know no proper connected subgraphs can occur since $\rho(G)\setminus\{b_1\}=\{p\}\cup\pi^*\cup\rho$ for the particular $\pi^*=\pi\setminus\{b_1\}$, and therefore the disconnected subgraph must occur. Following our usual argument, we get that $G/N$ must have a central Sylow $b_1$-subgroup by way of Theorem 5.5 from \cite{L2}, and so $O^{b_1}(G)<G$, a contradiction since $b_1$ has been shown to be admissible. In case (iii), we can form an identical argument as to that done in case (ii) since the vertices $b_1$ and $b_{k+1}$ are symmetric.

Suppose (d) occurs; that is, suppose $\rho(G)\setminus(\rho(G/N)\cap\rho(G/M))\subseteq\rho$. We can follow the argument given in \cite{LM} or \cite{Lewis}. Observe that $E$ has a Hall $\rho$-subgroup of $H$, and that $|E|$ is divisible by only those primes in $\rho$. Therefore, $E$ is the Hall $\rho$-subgroup of $H$. Next, there exists a character $\chi\in\Irr(G)$ such that all the primes in $\rho$ divide $\chi(1)$. Taking $\theta$ as an irreducible constituent of $\chi_{FE}$, we note that $\chi(1)/\theta(1)$ divides $|G:FE|$ and that $\chi(1)$ is relatively prime to $|G:FE|$. This forces $\chi_{FE}=\theta$. Next, the only possible divisors of $\cd(G/FE)$ are those primes in $\{p\}\cup\pi$, and we can apply Gallagher's Theorem to get that $\cd(G/FE)=\{1\}$, and therefore $G/FE$ is abelian. We now have that $O^b(G)<G$ for some $b\in\pi$, which is a contradiction since all $b\in\pi$ are admissible by Lemma \ref{admissiblevertices} and therefore $O^b(G)=G$.

Hence, no such $N$ can occur, and we get our desired conclusion that the Fitting subgroup $F$ is minimal normal in $G$.
\end{proof}

Now that we have concluded that the Fitting subgroup is minimal normal, we are now ready to apply Lemma \ref{lemma3}.

\begin{lemma}\label{finalpart}
Let $k\geq3$. The graph $\SRi{k}{1}{2}$ is not the prime character degree graph of any solvable group.
\end{lemma}
\begin{proof}
For the sake of contradiction, suppose that $G$ is a counterexample with $|G|$ minimal such that $\Delta(G)=\SRi{k}{1}{2}$. Observe that the Fitting subgroup of $G$ is minimal normal in $G$ by Lemma \ref{minimalnormal}. Next, notice that the conditions from Lemma \ref{lemma3} are satisfied by taking $a=a_1$, $b=a_2$, $c=b_1$, and $d=c_1$, where we observe that $b_1$ and $c_1$ are admissible by Lemma \ref{admissiblevertices}. Applying Lemma \ref{lemma3} yields our contradiction, and therefore the graph $\SRi{k}{1}{2}$ is not the prime character degree graph of any solvable group.
\end{proof}

\subsection{Proof of Theorem 1.1.}\label{ISsubsection}

Next, we assume our inductive hypothesis below, where we recall that we always take $k \geq \max{\{t,3\}}$.

\begin{hyp}\label{mainhyp}
Given any integer $t\geq1$, we assume that the graph $\SR{k}{t}$ does not occur as the prime character degree graph of any solvable group.
\end{hyp}

In order to tame the graph $\SR{k}{t+1}$, which is the goal, there is one particular subgraph that must be handled first: $\SRi{k}{t}{2}$. However, this is a rather tedious task, as we must first begin with the family of graphs classified in \cite{BL}. In particular, as consequence of Theorem \ref{KT} (using the graphs $\Gamma_{k+t+1,k}$ and $\Gamma_{k+t+2,k}$), the graph $\SRi{k}{1}{(t+1)}$ does not occur as $\Delta(G)$ for any solvable group $G$. The argument follows identically to what was done in the series of lemmas for $\SRi{k}{1}{2}$ (that is, Lemmas \ref{admissiblevertices} through \ref{finalpart}). Next, since the graph $\SRi{k}{1}{(t+1)}$ does not occur, this then yields the non-occurrence of $\SRi{k}{2}{t}$. Furthermore, this implies $\SRi{k}{3}{(t-1)}$ does not occur, and continuing in this way, we get to the non-occurrence of $\SRi{k}{t-1}{3}$. Using the aforementioned graph, along with the non-occurrence of $\SR{k}{t}$ (from Hypothesis \ref{mainhyp}), we can ultimately conclude that the graph $\SRi{k}{t}{2}$ does not occur as $\Delta(G)$ for any solvable group $G$. All these arguments again follow what was done in Subsection \ref{BCsubsection}. See Figure \ref{figchain} for the chain of implications leading to the non-occurrence of $\SRi{k}{t}{2}$.
\begin{figure}[htb]
    \centering
$\begin{matrix}
\Gamma_{k+t+1,k} & & & &\\
& \implies & \SRi{k}{1}{(t+1)} \implies \SRi{k}{2}{t} \implies \SRi{k}{3}{(t-1)} \implies \cdots \implies & \SRi{k}{t-1}{3} & &\\
\Gamma_{k+t+2,k} & & & & \implies & \SRi{k}{t}{2}\\
& & & \SR{k}{t} & &\\
\end{matrix}$
    \caption{Sequence of non-occurrence implications}
    \label{figchain}
\end{figure}

Finally, we can then proceed to the inductive step:

\begin{proposition}\label{mainstep}
The graph $\SR{k}{t+1}$ does not occur as the prime character degree graph of any solvable group.
\end{proposition}
\begin{proof}
We proceed by contradiction where we suppose there exists some solvable group $G$ with $|G|$ minimal such that $\Delta(G)=\SR{k}{t+1}$. One can verify that $b_i$ is strongly admissible for all $1\leq i\leq k+t+1$. In particular, $b_1,\ldots,b_{t+1},b_{k+1},\ldots,b_{k+t+1}$ all behave identically and rely on the graphs $\SR{k}{t}$ (Hypothesis \ref{mainhyp}) and $\SRi{k}{t}{2}$ (Figure \ref{figchain}), and $b_{t+2},\ldots,b_k$ are symmetrical and rely on diameter three arguments outlined in \cite{Sass}. Next, $a_1,\ldots,a_k$ all easily satisfy Lemma \ref{pi}. Finally, $\SR{k}{t+1}$ satisfies the hypothesis from \cite{BLL} under the notation corresponding to $p=c$ (which relies on the graph $\Gamma_{k+t+1,k}$ and the graphs $\SRi{k}{t+1-i}{(i+1)}$ for all $1\leq i\leq t$). Hence, there is no normal nonabelian Sylow $q$-subgroup for any vertex $q\in\rho(G)$.

As before, one can go through steps to get the Frattini subgroup $\Phi(G)=1$. Then one gets the existence of a subgroup $H$ of $G$ such that $G=HF$ and $H\cap F=1$, where $F$ is the Fitting subgroup of $G$. It is then an easy verification that $F$ is minimal normal in $G$, and then that $\SR{k}{t+1}$ satisfies the conditions of Lemma \ref{lemma3}, obtaining our contradiction.
\end{proof}

For examples of the next family of graphs handled by Proposition \ref{mainstep} (namely, the case for $n=2$), see Figure \ref{figk2R}.

\begin{figure}[htb]
    \centering
$
\begin{tikzpicture}[scale=2]
\node (1a) at (.5,1) {$a_1$};
\node (2a) at (.5,0) {$a_2$};
\node (11a) at (1.75,1) {$b_1$};
\node (22a) at (1.75,0) {$b_2$};
\node (33a) at (1.25,.75) {$b_3$};
\node (44a) at (1.25,.25) {$b_4$};
\node (55a) at (2.25,.5) {$c$};
\path[font=\small,>=angle 90]
(1a) edge node [right] {$ $} (2a)
(11a) edge node [right] {$ $} (22a)
(11a) edge node [right] {$ $} (33a)
(11a) edge node [right] {$ $} (44a)
(11a) edge node [right] {$ $} (55a)
(22a) edge node [right] {$ $} (33a)
(22a) edge node [right] {$ $} (44a)
(22a) edge node [right] {$ $} (55a)
(33a) edge node [right] {$ $} (44a)
(33a) edge node [right] {$ $} (55a)
(44a) edge node [right] {$ $} (55a)
(1a) edge node [right] {$ $} (11a)
(2a) edge node [right] {$ $} (22a)
(1a) edge node [right] {$ $} (33a)
(2a) edge node [right] {$ $} (44a);
\node (1b) at (3,1) {$a_1$};
\node (2b) at (3,0) {$a_2$};
\node (3b) at (3.5,.5) {$a_3$};
\node (11b) at (4.75,1) {$b_1$};
\node (22b) at (4.75,0) {$b_2$};
\node (33b) at (4.25,.5) {$b_3$};
\node (44b) at (5.25,.85) {$b_4$};
\node (55b) at (5.25,.15) {$b_5$};
\node (66b) at (5.75,.5) {$c$};
\path[font=\small,>=angle 90]
(1b) edge node [right] {$ $} (2b)
(1b) edge node [right] {$ $} (3b)
(2b) edge node [right] {$ $} (3b)
(11b) edge node [right] {$ $} (22b)
(11b) edge node [right] {$ $} (33b)
(11b) edge node [right] {$ $} (44b)
(11b) edge node [right] {$ $} (55b)
(11b) edge node [right] {$ $} (66b)
(22b) edge node [right] {$ $} (33b)
(22b) edge node [right] {$ $} (33b)
(22b) edge node [right] {$ $} (44b)
(22b) edge node [right] {$ $} (55b)
(22b) edge node [right] {$ $} (66b)
(33b) edge node [right] {$ $} (44b)
(33b) edge node [right] {$ $} (55b)
(33b) edge node [right] {$ $} (66b)
(44b) edge node [right] {$ $} (55b)
(44b) edge node [right] {$ $} (66b)
(55b) edge node [right] {$ $} (66b)
(1b) edge node [right] {$ $} (11b)
(2b) edge node [right] {$ $} (22b)
(3b) edge node [right] {$ $} (33b)
(1b) edge node [right] {$ $} (44b)
(2b) edge node [right] {$ $} (55b);
\node (1c) at (-.5,-.5) {$a_1$};
\node (2c) at (-.5,-1.5) {$a_2$};
\node (3c) at (0,-.75) {$a_3$};
\node (4c) at (0,-1.25) {$a_4$};
\node (11c) at (1.25,-.5) {$b_1$};
\node (22c) at (1.25,-1.5) {$b_2$};
\node (33c) at (.75,-.75) {$b_3$};
\node (44c) at (.75,-1.25) {$b_4$};
\node (55c) at (1.75,-.65) {$b_5$};
\node (66c) at (1.75,-1.35) {$b_6$};
\node (77c) at (2.25,-1) {$c$};
\path[font=\small,>=angle 90]
(1c) edge node [right] {$ $} (2c)
(1c) edge node [right] {$ $} (3c)
(1c) edge node [right] {$ $} (4c)
(2c) edge node [right] {$ $} (3c)
(2c) edge node [right] {$ $} (4c)
(3c) edge node [right] {$ $} (4c)
(11c) edge node [right] {$ $} (22c)
(11c) edge node [right] {$ $} (33c)
(11c) edge node [right] {$ $} (44c)
(11c) edge node [right] {$ $} (55c)
(11c) edge node [right] {$ $} (66c)
(11c) edge node [right] {$ $} (77c)
(22c) edge node [right] {$ $} (33c)
(22c) edge node [right] {$ $} (33c)
(22c) edge node [right] {$ $} (44c)
(22c) edge node [right] {$ $} (55c)
(22c) edge node [right] {$ $} (66c)
(22c) edge node [right] {$ $} (77c)
(33c) edge node [right] {$ $} (44c)
(33c) edge node [right] {$ $} (55c)
(33c) edge node [right] {$ $} (66c)
(33c) edge node [right] {$ $} (77c)
(44c) edge node [right] {$ $} (55c)
(44c) edge node [right] {$ $} (66c)
(44c) edge node [right] {$ $} (77c)
(55c) edge node [right] {$ $} (66c)
(55c) edge node [right] {$ $} (77c)
(66c) edge node [right] {$ $} (77c)
(1c) edge node [right] {$ $} (11c)
(2c) edge node [right] {$ $} (22c)
(3c) edge node [right] {$ $} (33c)
(4c) edge node [right] {$ $} (44c)
(1c) edge node [right] {$ $} (55c)
(2c) edge node [right] {$ $} (66c);
\node (1d) at (3.5,-.5) {$a_1$};
\node (2d) at (3.5,-1.5) {$a_2$};
\node (3d) at (4,-.75) {$a_3$};
\node (4d) at (4,-1.25) {$a_4$};
\node (5d) at (3,-1) {$a_5$};
\node (11d) at (5.75,-.5) {$b_1$};
\node (22d) at (5.75,-1.5) {$b_2$};
\node (33d) at (5.25,-.65) {$b_3$};
\node (44d) at (5.25,-1.35) {$b_4$};
\node (55d) at (4.75,-1) {$b_5$};
\node (66d) at (6.25,-.65) {$b_6$};
\node (77d) at (6.25,-1.35) {$b_7$};
\node (88d) at (6.75,-1) {$c$};
\path[font=\small,>=angle 90]
(1d) edge node [right] {$ $} (2d)
(1d) edge node [right] {$ $} (3d)
(1d) edge node [right] {$ $} (4d)
(1d) edge node [right] {$ $} (5d)
(2d) edge node [right] {$ $} (3d)
(2d) edge node [right] {$ $} (4d)
(2d) edge node [right] {$ $} (5d)
(3d) edge node [right] {$ $} (4d)
(3d) edge node [right] {$ $} (5d)
(4d) edge node [right] {$ $} (5d)
(11d) edge node [right] {$ $} (22d)
(11d) edge node [right] {$ $} (33d)
(11d) edge node [right] {$ $} (44d)
(11d) edge node [right] {$ $} (55d)
(11d) edge node [right] {$ $} (66d)
(11d) edge node [right] {$ $} (77d)
(11d) edge node [right] {$ $} (88d)
(22d) edge node [right] {$ $} (33d)
(22d) edge node [right] {$ $} (44d)
(22d) edge node [right] {$ $} (55d)
(22d) edge node [right] {$ $} (66d)
(22d) edge node [right] {$ $} (77d)
(22d) edge node [right] {$ $} (88d)
(33d) edge node [right] {$ $} (44d)
(33d) edge node [right] {$ $} (55d)
(33d) edge node [right] {$ $} (66d)
(33d) edge node [right] {$ $} (77d)
(33d) edge node [right] {$ $} (88d)
(44d) edge node [right] {$ $} (55d)
(44d) edge node [right] {$ $} (66d)
(44d) edge node [right] {$ $} (77d)
(44d) edge node [right] {$ $} (88d)
(55d) edge node [right] {$ $} (66d)
(55d) edge node [right] {$ $} (77d)
(55d) edge node [right] {$ $} (88d)
(66d) edge node [right] {$ $} (77d)
(66d) edge node [right] {$ $} (88d)
(77d) edge node [right] {$ $} (88d)
(1d) edge node [right] {$ $} (11d)
(2d) edge node [right] {$ $} (22d)
(3d) edge node [right] {$ $} (33d)
(4d) edge node [right] {$ $} (44d)
(5d) edge node [right] {$ $} (55d)
(1d) edge node [right] {$ $} (66d)
(2d) edge node [right] {$ $} (77d);
\end{tikzpicture}
$
    \caption{Examples of graphs in the family $\{\SR{k}{2}\}$: $2\leq k\leq5$}
    \label{figk2R}
\end{figure}

\begin{remark}
As a singular point of interest, we note that we are able to tame the graph $\SR{3}{3}$ (see Figure \ref{fig33R})
. The disconnected subgraph with the same vertex set has components of size $a=3$ and $b=7$, in which case $b=7=2^3-1=2^a-1$; this does not violate P\'alfy's inequality from \cite{P2}, but does attain equality. In particular, this graph satisfies the technical hypothesis from \cite{BLL} (which is sufficient for our argument), in contrast to how the unknown graphs in Figure \ref{figRmaybe} do not. All other graphs $\SR{k}{n}$ with $k\geq3$ and $(k,n)\neq(3,3)$ easily violate P\'alfy's inequality.
\end{remark}

\begin{figure}[htb]
    \centering
$
\begin{tikzpicture}[scale=2]
\node (1a) at (.5,.5) {$a_1$};
\node (2a) at (0,1) {$a_2$};
\node (3a) at (0,0) {$a_3$};
\node (11a) at (1.25,.75) {$b_1$};
\node (22a) at (1.75,1) {$b_2$};
\node (33a) at (1.75,0) {$b_3$};
\node (44a) at (1.25,.25) {$b_4$};
\node (55a) at (2.25,.85) {$b_5$};
\node (66a) at (2.25,.15) {$b_6$};
\node (77a) at (2.75,.5) {$c$};
\path[font=\small,>=angle 90]
(1a) edge node [right] {$ $} (2a)
(1a) edge node [right] {$ $} (3a)
(2a) edge node [right] {$ $} (3a)
(11a) edge node [right] {$ $} (22a)
(11a) edge node [right] {$ $} (33a)
(11a) edge node [right] {$ $} (44a)
(11a) edge node [right] {$ $} (55a)
(11a) edge node [right] {$ $} (66a)
(11a) edge node [right] {$ $} (77a)
(22a) edge node [right] {$ $} (33a)
(22a) edge node [right] {$ $} (44a)
(22a) edge node [right] {$ $} (55a)
(22a) edge node [right] {$ $} (66a)
(22a) edge node [right] {$ $} (77a)
(33a) edge node [right] {$ $} (44a)
(33a) edge node [right] {$ $} (55a)
(33a) edge node [right] {$ $} (66a)
(33a) edge node [right] {$ $} (77a)
(44a) edge node [right] {$ $} (55a)
(44a) edge node [right] {$ $} (66a)
(44a) edge node [right] {$ $} (77a)
(55a) edge node [right] {$ $} (66a)
(55a) edge node [right] {$ $} (77a)
(66a) edge node [right] {$ $} (77a)
(1a) edge node [right] {$ $} (11a)
(2a) edge node [right] {$ $} (22a)
(3a) edge node [right] {$ $} (33a)
(1a) edge node [right] {$ $} (44a)
(2a) edge node [right] {$ $} (55a)
(3a) edge node [right] {$ $} (66a);
\end{tikzpicture}
$
    \caption{The graph $\SR{3}{3}$}
    \label{fig33R}
\end{figure}

To conclude, we once again state our main theorem:

\begin{thmmain}
The graph $\SR{k}{n}$ occurs as the prime character degree graph of a solvable group when $(k,n)=(1,1)$ (see Figure \ref{fig11R}), and possibly when $(k,n)\in\{(2,1),(2,2)\}$ (see Figure \ref{figRmaybe}). Otherwise $\SR{k}{n}$ does not occur as the prime character degree graph of any solvable group.
\end{thmmain}
\begin{proof}
Follows by the induction argument given through Proposition \ref{n1case}, Hypothesis \ref{mainhyp}, and Proposition \ref{mainstep}.
\end{proof}

\begin{corollary}
For all integers $k$ and $n$ such that $1\leq n\leq k$ and $k\geq3$, any proper connected subgraph of $\SR{k}{n}$ with the same vertex set is not the prime character degree graph of any solvable group.
\end{corollary}

\subsection{Final remarks.}

The family $\{\SLR{k}{n}\}$ was originally studied in \cite{LM} to provide examples of graphs that satisfy the technical hypothesis from \cite{BLL}. Another motivation was to build upon the construction of the family of graphs studied in \cite{BL}. As is often the case with classifications of this sort, use of admissibility presents a dependence on knowing the status of subgraphs. Therefore, if the yet-unclassified graphs with five and six vertices from \cite{Lewis} and \cite{BL2}, respectively, are determined to occur or not occur, we may be able to leverage this information to better understand the unknown graphs from Figure \ref{figRmaybe}. Finally, one can imagine obvious ways to generalize and expand upon our results in this paper.


\begin{thebibliography}{99}


\bibitem{BissLaub} Mark W. Bissler and Jacob Laubacher. Classifying families of character degree graphs of solvable groups. \emph{Int. J. Group Theory}, 8(4):37--46, 2019.

\bibitem{BL2} Mark W. Bissler, Jacob Laubacher, and Mark L. Lewis. Classifying character degree graphs with six vertices. \emph{Beitr. Algebra Geom.}, 60(3):499--511, 2019.

\bibitem{BLL} Mark W. Bissler, Jacob Laubacher, and Corey F. Lyons. On the absence of a normal nonabelian Sylow subgroup. \emph{Comm. Algebra}, 47(3):917--922, 2019.

\bibitem{BL} Mark W. Bissler and Mark L. Lewis. A family of graphs that cannot occur as character degree graphs of solvable groups. arXiv:1707.03020, 2017.

\bibitem{H} B. Huppert. \emph{Endliche gruppen. I}. Die Grundlehren der Mathematischen Wissenschaften, Band 134. Springer-Verlag, Berlin-New York, 1967.

\bibitem{H2} B. Huppert. Research in representation theory at Mainz (1984--1990). \emph{Progr. Math.}, 95:17--36, 1991.


\bibitem{LM} Jacob Laubacher and Mark Medwid. On prime character degree graphs occurring within a family of graphs. \emph{Comm. Algebra}, 49(4):1534--1547, 2021.

\bibitem{L2} Mark L. Lewis. Solvable groups whose degree graphs have two connected components. \emph{J. Group Theory}, 4(3):255--275, 2001.

\bibitem{LD3} Mark L. Lewis. A solvable group whose character degree graph has diameter 3. \emph{Proc. Amer. Math. Soc.}, 130(3):625--630, 2002.


\bibitem{Lewis} Mark L. Lewis. Classifying character degree graphs with 5 vertices. In \emph{Finite groups 2003}, pages 247--265. Walter de Gruyter, Berlin, 2004.

\bibitem{L3} Mark L. Lewis. An overview of graphs associated with character degrees and conjugacy class sizes in finite groups. \emph{Rocky Mountain J. Math.}, 38(1):175--211, 2008.


\bibitem{P} P\'eter P\'al P\'alfy. On the character degree graph of solvable groups. I. Three primes. \emph{Period. Math. Hungar.}, 36(1):61--65, 1998.

\bibitem{P2}  P\'eter P\'al P\'alfy. On the character degree graph of solvable groups. II. Disconnected graphs. \emph{Studia Sci. Math. Hungar.}, 38:339--355, 2001.

\bibitem{Sass} Catherine B. Sass. Character degree graphs of solvable groups with diameter three. \emph{J. Group Theory}, 19(6):1097--1127, 2016.



\end{thebibliography}
\end{document}